\documentclass[a4paper]{amsart}

\usepackage[lmargin=1in,rmargin=1in,tmargin=1in,bmargin=1in]{geometry}
\RequirePackage{amsmath} 
\RequirePackage{amssymb}
\usepackage{amscd,latexsym,amsthm,amsfonts,amssymb,amsmath,amsxtra}
\usepackage[colorlinks=true,urlcolor=blue,citecolor=blue]{hyperref}
\usepackage{color}
\usepackage[all]{xy}
\usepackage[OT2,T1]{fontenc}
\usepackage{bm}
\usepackage{mathtools}
\usepackage{ mathrsfs }
\usepackage{xcolor}
\usepackage{comment}
\usepackage{marginnote}
\usepackage{enumitem}
\usepackage{thmtools, thm-restate}

\DeclareSymbolFont{cyrletters}{OT2}{wncyr}{m}{n}
\DeclareMathSymbol{\Sha}{\mathalpha}{cyrletters}{"58}

\let\Re\undefined
\let\Im\undefined

\DeclareMathOperator{\Re}{Re}
\DeclareMathOperator{\Im}{Im}

\DeclareMathOperator{\Spec}{Spec}

\newcommand{\bZ}{\mathbb{Z}}

\newcommand{\bQ}{\mathbb{Q}}
\newcommand{\bA}{\mathbb{A}}

\newcommand{\cF}{\mathcal{F}}

\newcommand{\bC}{\mathbb{C}}

\newcommand{\cS}{\mathcal{S}}

\newcommand{\bH}{\mathbb{H}}

\def\Re{\operatorname{Re}}

	\newcommand{\Mod}[1]{\ (\mathrm{mod}\ #1)}
        \newcommand{\sym}{\operatorname{sym}}

	\newcommand{\Geo}{\operatorname{Geo}}
	
	\newcommand{\K}{\operatorname{K}}
	
	\newcommand{\s}{\operatorname{sp}}
	\newcommand{\du}{\operatorname{Dual}}

	\newcommand{\Reg}{\operatorname{Reg}}

	\newcommand{\si}{\operatorname{Sing}}

	\newcommand{\sm}{\operatorname{Small}}

	\newcommand{\RNum}[1]{\uppercase\expandafter{\romannumeral #1\relax}}

\begin{document}
\theoremstyle{plain}
	\newtheorem{thm}{Theorem}[section]
	\newtheorem{cor}[thm]{Corollary}
	\newtheorem{thmy}{Theorem}
        \newtheorem{cory}{Corollary}
	\renewcommand{\thethmy}{\Alph{thmy}}
	\newenvironment{thmx}{\stepcounter{thm}\begin{thmy}}{\end{thmy}}
        \newenvironment{corx}{\stepcounter{cor}\begin{cory}}{\end{cory}}

	\renewcommand{\thecory}{\Alph{cory}}
	\newtheorem{hy}[thm]{Hypothesis}
	\newtheorem*{thma}{Theorem A}
	\newtheorem*{corb}{Corollary B}
	\newtheorem*{thmc}{Theorem C}
        \newtheorem*{thmd}{Theorem D}
	\newtheorem{lemma}[thm]{Lemma}  
	\newtheorem{prop}[thm]{Proposition}
	\newtheorem{conj}[thm]{Conjecture}
	\newtheorem{fact}[thm]{Fact}
	\newtheorem{claim}[thm]{Claim}
	
	\theoremstyle{definition}
	\newtheorem{defn}[thm]{Definition}
	\newtheorem{example}[thm]{Example}
	\theoremstyle{remark}
	
	\newtheorem{remark}[thm]{Remark}	
	\numberwithin{equation}{section}
	

\title[]{A note on a classical relative trace formula}%
\author{Zhining Wei}

\address{Kassar House, 151 Thayer St, Providence, RI 02912 USA}
\email{zhining$\_$wei@brown.edu}

\begin{abstract}
In this note, we derive a relative trace formula (RTF) using classical methods. We obtain a closed formula for the second moment of the central values of holomorphic cusp forms, a result originally established in Kuznetsov's preprint.
\end{abstract}

\date{\today}
\maketitle


\section{Introduction}

In this note, we will derive a relative trace formula (RTF) using classical methods. Our goal is to obtain the closed formula for the eigenvalue-weighted second moment of the central values of holomorphic cusp forms, first established in Kuznetsov's preprint. This result was independently proven by Iwaniec and Sarnak \cite[Theorem 17]{IwaniecSarnak2000}. For our purposes, we use the version presented in \cite[Theorem 4.2]{BF21}.

Let $k>2$ be an even number. Denote by $\cS_k(1)$ the space of holomorphic cusp forms of weight $k$ on $\Gamma=\operatorname{SL}_2(\bZ)$. Let $H_k$ be the Hecke orthogonal basis for $\cS_k(1)$. (This will force $k\geq12.$) We further assume that, for each $f\in H_k,$ $f$ has its Fourier expansion:
\[f(z)=\sum_{n=1}^{\infty}\lambda_f(n)n^{\frac{k-1}{2}}e(nz)\]
with $\lambda_f(1)=1.$ Let $\|f\|$ be the Petersson norm of $f$ and by the Rankin-Selberg theory,
\begin{equation}\label{norm and sym}
    \frac{1}{\|f\|^2}=\frac{(4\pi)^{k-1}2\pi^2}{\Gamma(k)L(1,\sym^2f)}.
\end{equation}

Let $s\in\bC$ with $\Re(s)\gg1.$ The $L$-function, denoted by $L(s,f)$, is defined to be:
\[L(s,f)=\sum_{n=1}^{\infty}\frac{\lambda_f(n)}{n^s}\]
and a direct calculation will show:
\begin{equation}\label{L function as an integral}
(2\pi)^{-(s+\frac{k}{2})}\Gamma\left(s+\frac{k}{2}\right)L\left(s+\frac{1}{2},f\right)=\int_0^{\infty}f(iy)y^{\frac{k}{2}+s}\frac{\,dy}{y}.
\end{equation}

Let $n\geq1$ be an integer. In this note, we will establish a RTF for the weighted second moment
\begin{equation}\label{weighted second moment}    
\sum_{f\in H_k}\frac{\lambda_f(n)L(1/2+s_1,f)L(1/2+s_2,f)}{L(1,\sym^2f)}.
\end{equation}
The main theorem is stated as follows, and it also appears as Theorem 4.2 in \cite{BF21}: 
\begin{thmx}\label{main theorem, weighted moment}
    (Kuznetsov, preprint $1994$) Assume that $\textbf{s}=(s_1,s_2)$ is in the region:
\begin{align*}
\begin{cases}
|\Re(s_1)|<\frac{k}{2}-1, \\ |\Re(s_2)|<\frac{k}{2}-1, 
\end{cases}
\end{align*}
and $s_1-s_2$ is not an integer. Then
\[\frac{2\pi^2}{k-1}\sum_{f\in H_k}\frac{\lambda_f(n)L(1/2+s_1,f)L(1/2+s_2,f)}{L(1,\sym^2f)}=M_2(n;s_1,s_2)+E(n;s_1,s_2),\]
where
\begin{align*}
M_2(n;s_1,s_2)&=\frac{\sigma_{s_1-s_2}(n)}{n^{s_1+1/2}}\zeta(1+s_1+s_2)+\frac{\sigma_{s_1-s_2}(n)}{n^{-s_2+1/2}}\zeta(1-s_1-s_2)\frac{\Gamma(-s_1+k/2)\Gamma(-s_2+k/2)}{(2\pi)^{-2s_1-2s_2}\Gamma(s_1+k/2)\Gamma(s_2+k/2)}\\
    &\hspace{4mm}+(-1)^{k/2}\frac{\sigma_{s_1+s_2}(n)}{n^{s_2+1/2}}\zeta(1-s_1+s_2)\frac{\Gamma(-s_1+k/2)}{(2\pi)^{-2s_1}\Gamma(s_1+k/2)}\\
    &\hspace{4mm}+(-1)^{k/2}\frac{\sigma_{s_1+s_2}(n)}{n^{s_1+1/2}}\zeta(1+s_1-s_2)\frac{\Gamma(-s_2+k/2)}{(2\pi)^{-2s_2}\Gamma(s_2+k)}.
\end{align*}
and 
\begin{align*}
E(n;s_1,s_2)&=\frac{(-1)^{k/2}}{n^{1/2}}\sum_{m=1}^{\infty}\sigma_{s_1-s_2}(m)\sigma_{s_1+s_2}(n+m)\psi_k\left({n},{m};s_1,s_2\right) \\
    &\hspace{4mm}+\frac{1}{n^{1/2}}\sum_{1\leq m\leq n-1}\sigma_{s_1+s_2}(n-m)\sigma_{s_1-s_2}(m)\phi_k(n,m;s_1,s_2)\\
    &\hspace{4mm}+\frac{(-1)^{k/2}}{n^{1/2}}\sum_{m\geq n+1}^{\infty}\sigma_{s_1-s_2}(m)\sigma_{s_1+s_2}(m-n)\Phi_k\left({n},{m};s_1,s_2\right).\\
\end{align*}
Here:
\begin{align*}
\psi_k\left({n},{m};s_1,s_2\right)&=2(2\pi)^{s_1+s_2}\cos\left(\frac{\pi}{2}(s_1-s_2)\right)\frac{m^{s_2}}{(n+m)^{s_1+s_2}}\left(\frac{n}{m}\right)^{k/2}\\
&\hspace{10mm}\times \frac{\Gamma(-s_1+k/2)\Gamma(-s_2+k/2)}{\Gamma(k)}F(-s_2+k/2,-s_1+k/2;k;-n/m);\\
{\phi}_k(n,m;s_1,s_2)&=\frac{(2\pi)^{s_1+s_2+1}}{2\sin\left(\frac{\pi}{2}(s_2-s_1)\right)}\frac{n^{s_2}}{(n-m)^{s_1+s_2}}\\
&\hspace{20mm}\times\frac{\Gamma(k/2-s_2)}{\Gamma(1+s_1-s_2)\Gamma(s_2+k/2)}F(k/2-s_2,1-k/2-s_2;s_1-s_2+1;m/n) \\
&\hspace{2mm}+\frac{(2\pi)^{s_1+s_2+1}}{2\sin\left(\frac{\pi}{2}(s_1-s_2)\right)}\frac{n^{s_1}m^{s_2-s_1}}{(n-m)^{s_1+s_2}}\\
&\hspace{20mm}\times\frac{\Gamma(k/2-s_1)}{\Gamma(1-s_1+s_2)\Gamma(s_1+k/2)}F(k/2-s_1,1-k/2-s_1;-s_1+s_2+1;m/n)\\
\Phi_k\left({n},{m};s_1,s_2\right)&=2(2\pi)^{s_1+s_2}\cos\left(\frac{\pi}{2}(s_1+s_2)\right)\frac{m^{s_2}}{(m-n)^{s_1+s_2}}\left(\frac{n}{m}\right)^{k/2}\\
&\hspace{10mm}\times  \frac{\Gamma(-s_1+k/2)\Gamma(-s_2+k/2)}{\Gamma(k)}F(-s_2+k/2,-s_1+k/2;k;n/m).\\
\end{align*}
Moreover, $(s_1,s_2)=(0,0)$ is a removable singularity.
\end{thmx}

The proof of Theorem \ref{main theorem, weighted moment} can be found in \S \ref{proof of main theorem and main corollary}. Additionally, we are particularly interested in the case $(s_1,s_2)=(0,0),$ which is Corollary \ref{main corollary. reduce to 0}. This corollary will be outlined in \S \ref{proof of main theorem and main corollary}.

The idea is based on the recent work \cite{WeiYangZhao2024}, but we will approach the problem in a purely classical manner.

\textbf{Acknowledgments} I am deeply grateful to Liyang Yang for his valuable suggestions and helpful discussions.

\section{Preliminaries}

Let $n\geq 1$. For $f\in\cS_k(1),$ we define the Hecke operator $T(n)$ by:
\[T(n)f=\sum_{\gamma\in  \Gamma\backslash G_n}f|_k\gamma=\frac{1}{n}\sum_{ad=n}a^{k}\sum_{b\Mod{d}}f\left(\frac{az+b}{d}\right),\]
where $G_n$ contains integral matrices of determinant $n.$ (An explicit definition is given in \S \ref{sec, geometric side}.) If $f\in H_k,$ then $T(n)f=\lambda_f(n)n^{\frac{k-1}{2}}f.$

On the other hand, let $\cF_k$ be an orthnormal basis for $\cS_k(1).$ By \cite[Theorem 1]{Zagier1977},
\[\sum_{f\in\cF_k}(T(n)f)(z)f(z')=C_k^{-1}n^{k-1}h_{k,n}(z,z')\]
where 
\[C_k=\frac{\pi i^k}{2^{k-3}(k-1)}\]
and
\[h_{k,n}(z,z')=\sum_{ad-bc=n}(czz'+dz'+az+b)^{-k}\]
In particular, taking $z=iy$, $z=iy'$ and $\cF_k=H_k$, we obtain: 
\begin{equation}\label{pre-trace}
\sum_{f\in H_k}\frac{(T(n)f)(iy_1)\overline{f(iy_2)}}{\|f\|^2}=C_k^{-1}n^{k-1}\sum_{ad-bc=n}(-cy_1y_2+idy_2+iay_1+b)^{-k}
\end{equation}
\begin{remark}
   In \cite{WeiYangZhao2024}, an adelic kernel function $\K(g_1,g_2)$ is defined, where $g_1,g_2\in\operatorname{GL}_2(\bA_F)$ and $F$ a totally real number field. It can be regarded as the adelic lift of $h_{k,n}(z,z')$ (up to some constant) via the strong approximation. 
\end{remark}
This motivates us, assuming $\mathbf{s}=(s_1,s_2)\in\bC^2$, to consider the function
\begin{equation}\label{relative trace formula function}
J(\mathbf{s},n)=C_k^{-1}n^{k-1}\int_0^{\infty}\int_0^{\infty}\sum_{ad-bc=n}\frac{y_1^{s_1+k/2}y_2^{s_2+k/2}}{(-cy_1y_2+idy_2+iay_1+b)^k}\frac{\,dy_1}{y_1}\frac{\,dy_2}{y_2}.
\end{equation}

\subsection{The spectral side}
Let $\Re(s_1), \Re(s_2)\gg1.$ Substituting \eqref{pre-trace} into \eqref{relative trace formula function} and sapping integrals by \eqref{L function as an integral}, we obtain 
\begin{equation*}
   J(\mathbf{s},n)=(2\pi)^{-(s_1+s_2+k)}n^{\frac{k-1}{2}}\Gamma(s_1+k/2)\Gamma(s_2+k/2)\sum_{f\in H_k}\frac{\lambda_f(n)L(1/2+s_1,f)L(1/2+s_2,f)}{\|f\|^2}
\end{equation*}
Insert \eqref{norm and sym} and we conclude:
\begin{equation}\label{spectral side of RTF}
   J(\mathbf{s},n)=\frac{2^{k-1}\pi }{(2\pi)^{s_1+s_2}}\frac{n^{\frac{k-1}{2}}\Gamma(s_1+k/2)\Gamma(s_2+k/2)}{\Gamma(k)}\sum_{f\in H_k}\frac{\lambda_f(n)L(1/2+s_1,f)L(1/2+s_2,f)}{L(1,\sym^2f)}
\end{equation}
For accuracy, we denote by $J_{\mathrm{Spec}}(\textbf{s},n)$ the integral \eqref{spectral side of RTF}, which is referred to as the \textit{spectral side} of the relative trace formula. This is absolutely convergent when $\Re(s_1),\Re(s_2)>\frac{1}{2}$ Obviously, we can take $s_1=s_2=0$ via the analytic continuation of automorphic $L$-functions.

\subsection{The geometric side}\label{sec, geometric side} 
We set
\[G_n=\left\{\begin{pmatrix}
    a&b\\c&d
\end{pmatrix}\in M_{2,2}(\bZ):ad-bc=n\right\}\]
and
\[R_{\gamma}(z,z'):=(\gamma z+z')j(\gamma,z).\]
We record the following important observation: for $\omega=\begin{pmatrix}
    &-1\\1&
\end{pmatrix}$, and any $\gamma\in\operatorname{GL}_2(\bQ)$,
\begin{equation}\label{identiy for R}
R_{\gamma\omega}(iy_1,iy_2)=R_{\gamma}(i/y_1,iy_2)(iy_1).
\end{equation}

By the defintion of $R_{\gamma}(z,z'),$ $J(\textbf{s},n)$ becomes:
\begin{equation}\label{geometric side of RTF}
J(\mathbf{s},n)=C_k^{-1}n^{k-1}\int_0^{\infty}\int_0^{\infty}\sum_{\gamma\in G_n}\frac{y_1^{s_1+k/2}y_2^{s_2+k/2}}{R_{\gamma}(iy_1,iy_2)^k}\frac{\,dy_1}{y_1}\frac{\,dy_2}{y_2} 
\end{equation}

 Let $J_{\mathrm{Geom}}(\textbf{s},n)$ denote the integral given by \eqref{geometric side of RTF}, which is conventionally referred to as the \textit{geometric side} of the relative trace formula. This is absolutely convergent when $\Re(s_1), \Re(s_2)\gg1.$ . To establish this, a ``regularization'' process is typically required, which is a central ingredient of RTF. (See \cite{Yan23a} or \cite[Theorem 3.2]{Yan23c}.) However, in our case, the regularization is simpler, and we will discuss it in \S \ref{sec. regularization}.

Denote by $\Gamma_{\infty}=\left\{\pm\begin{pmatrix}
    1&m\\&1
\end{pmatrix}:m\in\bZ\right\}$ and $\overline{\Gamma}_{\infty}=\left\{\begin{pmatrix}
    1&m\\&1
\end{pmatrix}:m\in\bZ\right\}.$ 

We set
\begin{align*}
 \Delta_n&=\left\{\begin{pmatrix}
    a&b\\&d
\end{pmatrix}:\mbox{$ad=n,$ $0\leq b<d$, $a,d>0$ and $a,b,d\in\bZ$}\right\},\\
B_n&=\left\{\begin{pmatrix}
    a&b\\&d
\end{pmatrix}:\mbox{$ad=n,$ $b\in \bZ$ and $a,b,d\in\bZ$}\right\}=\Gamma_{\infty}\Delta_n,\\
L_n&=\left\{\begin{pmatrix}
    a&\\b&d
\end{pmatrix}:\mbox{$ad=n,$ $0\neq b\in \bZ$ and $a,b,d\in\bZ$}\right\}\\
\overline{L}_n&=L_n/\{\pm I_2\}=\left\{\begin{pmatrix}
    a&\\b&d
\end{pmatrix}:\mbox{$ad=n,$ $0\neq b\in \bZ$ $a,d>0$ and $a,b,d\in\bZ$}\right\}
\end{align*}

Let $\omega=\begin{pmatrix}
    &-1\\1
\end{pmatrix}.$ We prove the following decomposition of $G_n:$

\begin{lemma}
The Bruhat decomposition for $G_n$ is:
\begin{align*}
    G_n=B_n\bigsqcup B_n\omega\bigsqcup L_n\bigsqcup L_n\omega\bigsqcup\left(\bigsqcup_{m\neq0,-n}\bigsqcup_{\substack{ad=n+m\\ a,d\in\bZ}}\bigsqcup_{\substack{bc=m\\ b,c\in\bZ}}\begin{pmatrix}
        a&b\\c&d
    \end{pmatrix}\right)
\end{align*}
\end{lemma}
\begin{proof}
    The first component contains matrices of the form $\begin{pmatrix}
        *&*\\&*
    \end{pmatrix};$

    The second component contains matrices of the form $\begin{pmatrix}
        *&*\\ *&
    \end{pmatrix};$

    The third component contains matrices of the form $\begin{pmatrix}
        *&\\ *&*
    \end{pmatrix},$ with bottom left corner being nonzero;

    The fourth component contains matrices of the form $\begin{pmatrix}
        &*\\ *&*
    \end{pmatrix}$ with bottom right corner being nonzero;

    The last component contains matrices with each entry being nonzero.
\end{proof}

Notice that $k$ is an even integer, we insert those decompositions into \eqref{geometric side of RTF} and we obtain:
\[J_{\Geo}(\textbf{s},n)=\sum_{\delta\in\{I_2,\omega\}}J_{\sm}^{\delta}(\textbf{s},n)+\sum_{\delta\in\{I_2,\omega\}}J_{\du}^{\delta}(\textbf{s},n)+J_{\Reg}(\textbf{s},n),\]
where
\begin{equation}\label{small orbit of RTF}
   J_{\sm}^{\delta}(\textbf{s},n)= 2C_k^{-1}n^{k-1}\int_0^{\infty}\int_0^{\infty}\sum_{\rho\in\Delta_n}\sum_{\gamma\in \overline{\Gamma}_{\infty}}\frac{y_1^{s_1+k/2}y_2^{s_2+k/2}}{R_{\gamma\rho\delta}(iy_1,iy_2)^k}\frac{\,dy_1}{y_1}\frac{\,dy_2}{y_2},
\end{equation}
\begin{equation}\label{dual orbit of RTF}
   J_{\du}^{\delta}(\textbf{s},n)= 2C_k^{-1}n^{k-1}\int_0^{\infty}\int_0^{\infty}\sum_{\gamma\in \overline{L}_{n}}\frac{y_1^{s_1+k/2}y_2^{s_2+k/2}}{R_{\gamma\delta}(iy_1,iy_2)^k}\frac{\,dy_1}{y_1}\frac{\,dy_2}{y_2},
\end{equation}
and
\begin{equation}\label{regular orbit of RTF}
   J_{\Reg}(\textbf{s},n)= C_k^{-1}n^{k-1}\int_0^{\infty}\int_0^{\infty}\sum_{m\neq0,-n}\sum_{\substack{ad=n+m\\ a,d\in\bZ}}\sum_{\substack{bc=m\\ b,c\in\bZ}}\frac{y_1^{s_1+k/2}y_2^{s_2+k/2}}{(-cy_1y_2+idy_2+iay_1+b)^k}\frac{\,dy_1}{y_1}\frac{\,dy_2}{y_2}.
\end{equation}
We will call $J_{\sm}^{\delta}(\textbf{s},n)$ the \textit{small cell orbital integrals}, $J_{\du}^{\delta}(\textbf{s},n)$ the \textit{dual orbital integrals} and $J_{\Reg}(\textbf{s},n)$ the \textit{regular orbital integral}. We also define
\[J_{\si}(\textbf{s},n)=\sum_{\delta\in\{I_2,\omega\}}J_{\sm}^{\delta}(\textbf{s},n)+\sum_{\delta\in\{I_2,\omega\}}J_{\du}^{\delta}(\textbf{s},n),\]
which is the \textit{singular orbital integrals}. 

\subsection{The complex logarithm and complex power functions} In this note, we assume that $z=re^{i\theta}\in\bC$ with $\theta\in(-\pi,\pi]$ and $r>0.$ Denote by $\log$ the complex logarithm function, and its domain is $\bC-(-\infty,0].$ 

\subsection{Integral representations for hypergeometric functions}
We record the following expressions for the hypergeomertric functions $F(\alpha,\beta;\gamma;z):= {}_{2}F_{1}(\alpha,\beta;\gamma;z)$:
\begin{itemize}
    \item \cite[Equation 3.197(1)]{GradshteynRyzhik2007}:
\begin{equation}\label{integral repn for hypergeometric}
\int_0^{\infty}x^{\nu}(\beta+x)^{-\mu}(x+\gamma)^{-\rho}\frac{\,dx}{x}=\beta^{-\mu}\gamma^{\nu-\rho}B(\nu,\mu-\nu+\rho)F(\mu,\nu;\mu+\rho;1-\gamma/\beta).
\end{equation}
This is valid when $|\arg(\beta)|,|\arg(\gamma)|<\pi$, $\Re(\nu)>0$ and $\Re(\mu)>\Re(\nu-\rho).$
\item \cite[3.197(2)]{GradshteynRyzhik2007}:
\begin{equation}\label{integral repen for hyper, lower u}
  \int_u^{\infty}x^{-\lambda+1}(x+\beta)^{\nu}(x-u)^{\mu-1} \frac{\,dx}{x}=u^{-\lambda}(\beta+u)^{\mu+\nu}B(\lambda-\mu-\nu,\mu)F(\lambda,\mu;\lambda-\nu;-\beta/u).  
\end{equation}
This is valid when $|\arg(u/\beta)|<\pi$ or $\left|\frac{\beta}{u}\right|<1$, $0<\Re(\mu)<\Re(\lambda-\nu)$.
\item \cite[Equation 3.197(8)]{GradshteynRyzhik2007}
\begin{equation}\label{integral repen for hyper, upper u}
  \int_0^{u}x^{\nu}(x+\alpha)^{\lambda}(u-x)^{\mu-1} \frac{\,dx}{x}=\alpha^{\lambda}u^{\mu+\nu-1}B(\mu,\nu)F(-\lambda,\nu;\mu+\nu;-u/\alpha). 
\end{equation}
This is valid when $|\arg(u/\alpha)|<\pi$, $\Re(\mu)>0$ and $\Re(\nu)>0.$
\end{itemize}

\subsection{The divisor functions} Let $v\in\bC.$ We define the divisor function
\[\sigma_v(n)=\sum_{d|n}d^v.\]
Notice that $\sigma_v(n)=n^{v}\sigma_{-v}(n).$ We also define
\[\tau_v(n)=\sum_{n_1n_2=n}n_1^vn_2^{-v}=\frac{\sigma_{2v}(n)}{n^v}.\]
Notice that $\tau_v(n)=\tau_{-v}(n)$ and for any $\varepsilon>0,$
\begin{equation}\label{bound for tau function}
    \tau_v(n)\ll n^{|\Re(v)|+\varepsilon}.
\end{equation}

\subsection{Some possible extensions}
Here are two possible extensions:
\begin{itemize}
\item 
Let $q\geq1$ be an integer and denote by $\Gamma_0(q)$ the Hecke congruence subgroup. Let $n$ be another integer such that $(n,q)=1.$ We can also define the kernel function
\[h_{k,n,q}(z,z')=\sum_{\substack{ad-bc=n\\ q|c}}(czz'+dz'+az+b)^{-k}=\sum_{\gamma\in G_{n}(q)}\frac{1}{j(\gamma,z)^k(\gamma z+z')^k},\]
with $G_{n}(q)=\left\{\begin{pmatrix}
    a&b\\c&d
\end{pmatrix}\in M_{2,2}(\bZ):ad-bc=n,\hspace{2mm}q|c\right\}.$ Following the argument in \cite[Theorem 1]{Zagier1977}, we can show:
\[\sum_{f\in\cF_k(q)}(T(n)f)(z)f(z')=C_k^{-1}n^{k-1}h_{k,n,q}(z,z'),\]
where $\cF_k(q)$ is an orthonormal basis for the holomorphic cusp forms of weight $k$ and level $q.$ (In this case, the Petersson norm is defined over $\Gamma_0(q)\backslash\bH.$) 

We can also investigate the RTF derived from this kernel function. In this case, the spectral side becomes more intricate since we need to exclude the oldforms. However, the geometric side simplifies, as matrices of the form $\begin{pmatrix}
    *&*\\ *
\end{pmatrix}$ and $\begin{pmatrix}
    &*\\ *&*
\end{pmatrix}$ will disappear (since $(n,q)=1$). 

\item
Next, assume that $q\geq1$ be an integer and $\chi$ a primitive character mod $q$ and denote by $\tau(\chi)$ its Gauss sum. Let $n\geq1$ be an integer satisfying $(n,q)=1.$ The well-known identity:
\[\tau(\chi)\overline{\chi(n)}=\sum_{u\Mod{q}}\chi(u)e\left(\frac{un}{q}\right)\] 
motivates us, assuming $\mathbf{s}=(s_1,s_2)\in\bC^2$, to consider the function
\[
J(\mathbf{s},n,\chi)=C_k^{-1}n^{k-1}\sum_{u_1,u_2\pmod{q}}\overline{\chi(u_1)}\chi(u_2)\int_0^{\infty}\int_0^{\infty}\sum_{\gamma\in G_n(q)}\frac{y_1^{s_1+k/2}y_2^{s_2+k/2}}{R_{\gamma}(iy_1+u_1/q,iy_2+u_2/q)^k}\frac{\,dy_1}{y_1}\frac{\,dy_2}{y_2}.
\]
Unfortunately, it is difficult to provide an explicit form for the spectral side, as the oldforms become more complicated when $q$ is not squarefree. However, we can show that
\[\sum_{q'|q}\sum_{f\in H_k^{\operatorname{new}}(q')}\lambda_f(n)|L(1/2,f\times\chi)|^2\ll J((0,0),n,\chi).\]
By analyzing the geometric side of the integral and applying the amplification method, we should be able to establish a Burgess-type bound in the character aspect, that is,
\begin{equation}\label{burgess bound in the character aspect}
 L(1/2,f\times \chi)\ll q^{3/8+\varepsilon}
\end{equation}

In the classical case, the Burgess-type bound \eqref{burgess bound in the character aspect} has been extensively studied (e.g., cf. \cite{Bykovskii1996}, \cite{BlomerHarcos2008}). This bound was later generalized to number fields using the relative trace formula method (cf. \cite{Yan23b}), which, in turn, supports our belief that a Burgess-type bound can be derived from $J(\textbf{s},n,\chi)$.

\end{itemize}

\section{The singular orbital integrals}

\subsection{The small cell orbital integrals}

In this section, we investigate \eqref{small orbit of RTF}:
\[ J_{\sm}^{\delta}(\textbf{s},n)= 2C_k^{-1}n^{k-1}\int_0^{\infty}\int_0^{\infty}\sum_{\rho\in\Delta_n}\sum_{\gamma\in \overline{\Gamma}_{\infty}}\frac{y_1^{s_1+k/2}y_2^{s_2+k/2}}{R_{\gamma\rho}(iy_1,iy_2)^k}\frac{\,dy_1}{y_1}\frac{\,dy_2}{y_2}.\]
The main result in this section is:

\begin{prop}\label{prop. small cell}
    Assume the notations before. 
\begin{itemize}
    \item 
 $ J_{\sm}^{I_2}(\textbf{s},n)$ is absolutely convergent in the region:
\begin{align*}
\begin{cases}
\Re(s_1+s_2)>0\\
\Re(s_1), \hspace{2mm}\Re(s_2)>-\frac{k}{2}. 
\end{cases}
\end{align*}
Moreover, it admits a meromorphic continuation to $\mathbb{C}^2$ given explicitly by 
\begin{equation}\label{main term 1}
 J_{\sm}^{I_2}(\textbf{s},n)= 2C_k^{-1}\frac{i^k\Gamma(s_1+k/2)\Gamma(s_2+k/2)}{(2\pi)^{s_1+s_2}\Gamma(k)}\zeta(1+s_1+s_2)\frac{\sigma_{s_1-s_2}(n)n^{k-1}}{n^{s_1+k/2}} 
\end{equation}

\item $ J_{\sm}^{\omega}(\textbf{s},n)$ is absolutely convergent in the region:
\begin{align*}
\begin{cases}
\Re(-s_1+s_2)>0\\
\Re(s_1)<\frac{k}{2}, \hspace{2mm}\Re(s_2)>-\frac{k}{2}. 
\end{cases}
\end{align*}
Moreover, it admits a meromorphic continuation to $\mathbb{C}^2$ given explicitly by 
\begin{equation}\label{main term 2}
 J_{\sm}^{\omega}(\textbf{s},n)= 2C_k^{-1}\frac{\Gamma(-s_1+k/2)\Gamma(s_2+k/2)}{(2\pi)^{-s_1+s_2}\Gamma(k)}\zeta(1-s_1+s_2)\frac{\sigma_{s_1+s_2}(n)n^{k-1}}{n^{s_2+k/2}}. 
\end{equation}
\end{itemize}

\end{prop}

\begin{proof}

We first study the case $\delta=I_2$. Notice that $\gamma=\begin{pmatrix}
    1&m\\&1
\end{pmatrix}\in\overline{\Gamma}_{\infty}$, we have:
\[R_{\gamma\rho}(z,z')=(\rho(z)+z'+m)j(\rho,z).\]
Insert it into $J_{\sm}^{I_2}(\textbf{s},n)$ and we obtain:
\[J_{\sm}^{I_2}(\textbf{s},n)= 2C_k^{-1}n^{k-1}\int_0^{\infty}\int_0^{\infty}\sum_{\rho\in\Delta_n}\frac{y_1^{s_1+k/2}y_2^{s_2+k/2}}{j(\rho,iy_1)^k}\sum_{m\in \bZ}\frac{1}{(\rho(iy_1)+iy_2+m)^k}\frac{\,dy_1}{y_1}\frac{\,dy_2}{y_2}.\]
Recall the following result (e.g., \cite[Equation (1.46)]{Iwaniec1997}): for $k\geq4,$
\begin{equation}\label{key identity}
\sum_{m\in\bZ}\frac{1}{(z+n)^k}=\frac{(-2\pi i)^k}{\Gamma(k)}\sum_{m\geq1}m^{k-1}e(mz).  
\end{equation}
Insert \eqref{key identity} into $J_{\sm}^{I_2}(\textbf{s},n)$ (notice that $j(\rho,iy_1)=d$ if the lower right corner of $\rho$ is $d$):
\begin{align*}
J_{\sm}^{I_2}(\textbf{s},n)&= 2C_k^{-1}n^{k-1}\frac{(-2\pi i)^k}{\Gamma(k)}\int_0^{\infty}\int_0^{\infty}\sum_{\rho\in\Delta_n}\frac{y_1^{s_1+k/2}y_2^{s_2+k/2}}{j(\rho,iy_1)^k}\sum_{m\geq1}m^{k-1}e^{2\pi im\rho(iy_1)}e^{-2\pi my_2}\frac{\,dy_1}{y_1}\frac{\,dy_2}{y_2}\\
&=2C_k^{-1}n^{k-1}\frac{(-2\pi i)^k}{\Gamma(k)}\sum_{m\geq1}m^{k-1}\left(\int_0^{\infty}y_2^{s_2+k/2}e^{-2\pi my_2}\frac{\,dy_2}{y_2}\right)\\
&\hspace{40mm}\times\sum_{d|n}\frac{1}{d^k}\left(\sum_{b\Mod{d}}\int_0^{\infty}y_1^{s_1+k/2}e^{-\frac{2\pi mny}{d^2}}e(mb/d)\frac{\,dy_1}{y_1}\right)
\end{align*}
Recall that, when $\Re(z)>0,$ $\int_0^{\infty}e^{-y}y^z\frac{\,dy}{y}=\Gamma(z).$ Thus, when $\Re(s_1)>-\frac{k}{2}$, $\Re(s_2)>-\frac{k}{2}$ and $\Re(s_1+s_2)>0$, $J_{\sm}^{I_2}(\textbf{s},n)$ is absolutely convergent and we obtain \eqref{main term 1}. It has a meromorphic continuation to $\bC^2$.

\bigskip

Next, we consider the case $\delta=\omega.$ Insert \eqref{identiy for R} into \eqref{small orbit of RTF} and do the change of variable $y_1\mapsto \frac{1}{y_1}$, one has:
\[ J_{\sm}^{\omega}(\textbf{s},n)= 2C_k^{-1}n^{k-1}i^k\int_0^{\infty}\int_0^{\infty}\sum_{\rho\in\Delta_n}\sum_{\gamma\in \overline{\Gamma}_{\infty}}\frac{y_1^{-s_1+k/2}y_2^{s_2+k/2}}{R_{\gamma\rho}(iy_1,iy_2)^k}\frac{\,dy_1}{y_1}\frac{\,dy_2}{y_2}.\]
Following the calculation of $J_{\sm}^{I_2}(\textbf{s},n)$, $J_{\sm}^{\omega}(\textbf{s},n)$ is absolutely convergent in the region
\[\Re(s_1)<\frac{k}{2}, \hspace{8mm} \Re(s_2)>-\frac{k}{2},\hspace{8mm} \Re(-s_1+s_2)>0\]
and one has:
\[
 J_{\sm}^{\omega}(\textbf{s},n)= 2C_k^{-1}\frac{\Gamma(-s_1+k/2)\Gamma(s_2+k/2)}{(2\pi)^{-s_1+s_2}\Gamma(k)}\zeta(1-s_1+s_2)\frac{\sigma_{-s_1-s_2}(n)n^{k-1}}{n^{-s_1+k/2}} 
\]
Recall that $\sigma_v(n)=n^{v}\sigma_{-v}(n)$ and \eqref{main term 2} follows. It also has a meromorphic continuation to $\bC^2$
\end{proof}

\subsection{The dual orbital integrals}
In this section, we study $J_{\du}^{\delta}(\textbf{s},n).$ The main result is:

\begin{prop}\label{prop. dual cell}
    Assume the notations before. 
\begin{itemize}
    \item 
 $ J_{\du}^{I_2}(\textbf{s},n)$ is absolutely convergent in the region:
\begin{align*}
\begin{cases}
\Re(s_1+s_2)>1\\
-\frac{k}{2}+1<\Re(s_1)<\frac{k}{2}, \hspace{2mm} \Re(s_2)<\frac{k}{2}. 
\end{cases}
\end{align*}
Moreover, it admits a meromorphic continuation to $\mathbb{C}^2$ given explicitly by 
\begin{equation}\label{main term 3}
J_{\du}^{I_2}(\textbf{s},n)=2(2\pi)^{s_1+s_2}C_k^{-1}i^k\frac{n^{k-1}\sigma_{s_1-s_2}(n)}{n^{-s_2+k/2}}\frac{\Gamma(-s_1+k/2)\Gamma(-s_2+k/2)}{\Gamma(k)}\zeta(1-s_1-s_2).
\end{equation}

\item $ J_{\du}^{\omega}(\textbf{s},n)$ is absolutely convergent in the region:
\begin{align*}
\begin{cases}
\Re(-s_1+s_2)>1\\
-\frac{k}{2}<\Re(s_1)<\frac{k}{2}-1, \hspace{2mm} \Re(s_2)<\frac{k}{2}. 
\end{cases}
\end{align*}
Moreover, it admits a meromorphic continuation to $\mathbb{C}^2$ given explicitly by 
\begin{equation}\label{main term 4}
    J_{\du}^{\omega}(\textbf{s},n)=2(2\pi)^{-s_1+s_2}C_k^{-1}\frac{n^{k-1}\sigma_{s_1+s_2}(n)}{n^{s_1+k/2}}\frac{\Gamma(s_1+k/2)\Gamma(-s_2+k/2)}{\Gamma(k)}\zeta(1+s_1-s_2).
\end{equation}
\end{itemize}

\end{prop}

\begin{proof}

We first consider the case $\delta=I_2:$
\[ J_{\du}^{I_2}(\textbf{s},n)= 2C_k^{-1}n^{k-1}\int_0^{\infty}\int_0^{\infty}\sum_{\gamma\in \overline{L}_{n}}\frac{y_1^{s_1+k/2}y_2^{s_2+k/2}}{R_{\gamma}(iy_1,iy_2)^k}\frac{\,dy_1}{y_1}\frac{\,dy_2}{y_2}.\]
By the definition of $L_n$ and $R_{\gamma}(iy_1,iy_2),$ this is
\begin{align*}    
J_{\du}^{I_2}(\textbf{s},n)&= 2C_k^{-1}n^{k-1}\int_0^{\infty}\int_0^{\infty}\sum_{ad=n}\sum_{\substack{b\in\bZ\\b\neq0}}\frac{y_1^{s_1+k/2}y_2^{s_2+k/2}}{(-by_1y_2+i(dy_2+ay_1))^k}\frac{\,dy_1}{y_1}\frac{\,dy_2}{y_2}\\
&= 2C_k^{-1}n^{k-1}\int_0^{\infty}\int_0^{\infty}\sum_{ad=n}\sum_{b=1}^{\infty}\frac{y_1^{s_1-k/2}y_2^{s_2-k/2}}{b^k\left(-1+i\left(\frac{d}{by_1}+\frac{a}{by_2}\right)\right)^k}\frac{\,dy_1}{y_1}\frac{\,dy_2}{y_2}\\
&\hspace{5mm}+2C_k^{-1}n^{k-1}\int_0^{\infty}\int_0^{\infty}\sum_{ad=n}\sum_{b=1}^{\infty}\frac{y_1^{s_1-k/2}y_2^{s_2-k/2}}{b^k\left(1+i\left(\frac{d}{by_1}+\frac{a}{by_2}\right)\right)^k}\frac{\,dy_1}{y_1}\frac{\,dy_2}{y_2}
\end{align*}
Assume that $\Re(s_1+s_2)>1,$ and do the change of variable $y_1\mapsto \frac{d}{by_1}$ and $y_2\mapsto\frac{a}{by_2},$ which yields: 
\begin{align*}
J_{\du}^{I_2}(\textbf{s},n)= 2C_k^{-1}\zeta(s_1+s_2)\frac{n^{k-1}\sigma_{s_1-s_2}(n)}{n^{-s_2+k/2}}(J(s_1,s_2,k)+\overline{J(\overline{s_1},\overline{s_2},k)})
\end{align*}
where
\[J(s_1,s_2,k)=\int_0^{\infty}\int_0^{\infty}\frac{y_1^{-s_1+k/2}y_2^{-s_2+k/2}}{(1+iy_1+iy_2)^k}\frac{\,dy_1}{y_1}\frac{\,dy_2}{y_2}.\]
Next, we do the change of variable $y_2\mapsto \frac{1}{y_2}$ and $y_1\mapsto \frac{y_1}{y_2}$:
\[J(s_1,s_2,k)=\int_0^{\infty}\int_0^{\infty}\frac{y_2^{s_1+s_2}y_1^{-s_1+k/2}}{(y_2+i(y_1+1))^k}\frac{\,dy_2}{y_2}\frac{\,dy_1}{y_1}.\]
Following the method in \cite[Proposition 2.2]{RR05}, the integral is convergent when $\Re(s_1)>-\frac{k}{2}+1$ and $\Re(s_2)<\frac{k}{2}.$

We first study the inner integral of $y_2.$ We define the function $h(z)=(z+i(y_1+1))^{-k}e^{(s_1+s_2-1)\log z}.$ This is a holomorphic function when $-\frac{1}{10}<\arg(z)<\frac{\pi}{2}+\frac{1}{10}.$ Let $R>0.$ We consider the sector contour constructed by the line segment connecting $0$ and $R,$ the arc connecting $R$ and $iR$ and the line segment connecting $iR$ and $0.$ By Cauchy's theorem, one has: when $R\to\infty,$
\[J(s_1,s_2,k)=\frac{e^{\frac{\pi i}{2}(s_1+s_2)}}{i^k}\int_0^{\infty}y_1^{-s_1+k/2}\int_0^{\infty}\frac{y_2^{s_1+s_2}}{(y_2+y_1+1)^k}\frac{\,dy_2}{y_2}\frac{\,dy_1}{y_1}.\]
Then the change of variable $y_2\mapsto y_2(y_1+1)$ implies:
\[J(s_1,s_2,k)=\frac{e^{\frac{\pi i}{2}(s_1+s_2)}}{i^k}\int_0^{\infty}y_1^{-s_1+k/2}(y_1+1)^{s_1+s_2-k}\int_0^{\infty}\frac{y_2^{s_1+s_2}}{(y_2+1)^k}\frac{\,dy_2}{y_2}\frac{\,dy_1}{y_1}.\]
Recall that, for $\Re(z_1),\Re(z_2)>0,$ we have
\begin{equation}\label{B function}
    \frac{\Gamma(z_1)\Gamma(z_2)}{\Gamma(z_1+z_2)}=B(z_1,z_2)=\int_0^{\infty}\frac{t^{z_1}}{(t+1)^{z_1+z_2}}\frac{\,dt}{t}.
 \end{equation}
This implies: when $\Re(s_1),\Re(s_2)<\frac{k}{2}$ and $\Re(s_1+s_2)>0,$
\[J(s_1,s_2,k)=\frac{e^{\frac{\pi i}{2}(s_1+s_2)}}{i^k}\frac{\Gamma(-s_1+k/2)\Gamma(-s_2+k/2)\Gamma(s_1+s_2)}{\Gamma(k)}.\]
Insert it into $J_{\du}^{I_2}(\textbf{s},n)$:
\[J_{\du}^{I_2}(\textbf{s},n)=4C_k^{-1}i^k\frac{n^{k-1}\sigma_{s_1-s_2}(n)}{n^{-s_2+k/2}}\frac{\Gamma(-s_1+k/2)\Gamma(-s_2+k/2)}{\Gamma(k)}\Gamma(s_1+s_2)\cos\left(\frac{\pi}{2}(s_1+s_2)\right)\zeta(s_1+s_2).\]
Recall the functional equation for Riemann zeta function:
\[\zeta(1-s)=\pi^{-s}2^{1-s}\Gamma(s)\cos\left(\frac{\pi s}{2}\right)\zeta(s)\]
and we conclude \eqref{main term 3}.

When $\delta=\omega,$ we apply \eqref{identiy for R} and change of variable $y_1\mapsto \frac{1}{y_1}:$
\[J_{\du}^{\omega}(\textbf{s},n)= 2C_k^{-1}n^{k-1}i^k\int_0^{\infty}\int_0^{\infty}\sum_{\gamma\in \overline{L}_{n}}\frac{y_1^{-s_1+k/2}y_2^{s_2+k/2}}{R_{\gamma}(iy_1,iy_2)^k}\frac{\,dy_1}{y_1}\frac{\,dy_2}{y_2}.\]
The argument in $\delta=I_2$ case derives:
\[J_{\du}^{\omega}(\textbf{s},n)=4C_k^{-1}i^k\frac{n^{k-1}\sigma_{-s_1-s_2}(n)}{n^{-s_2+k/2}}\frac{\Gamma(s_1+k/2)\Gamma(-s_2+k/2)}{\Gamma(k)}\Gamma(-s_1+s_2)\cos\left(\frac{\pi}{2}(-s_1+s_2)\right)\zeta(-s_1+s_2)\]

By the functional equation for Riemann zeta function and  $\sigma_v(n)=n^{v}\sigma_{-v}(n)$, we conclude \eqref{main term 4}.

\end{proof}

\subsection{The singular orbital integrals: a summary}
In this section, we collect all the results of singular orbital integrals.  

\begin{prop}\label{prop. singular}
 The singular orbital integral $J_{\si}(\textbf{s},n)$ is absolutely convergent in the region:
    \begin{align*}
\begin{cases}
\Re(s_1+s_2)>1, \hspace{2mm} \Re(-s_1+s_2)>1\\
|\Re(s_1)|<\frac{k}{2}-1, \hspace{2mm} |\Re(s_2)|<\frac{k}{2}. 
\end{cases}
\end{align*}
It has a meromorphic continuation to $\bC^2$ given explicitly by
\begin{align*}
    J_{\si}(\textbf{s},n)&=2C_k^{-1}\frac{i^k\Gamma(s_1+k/2)\Gamma(s_2+k/2)}{(2\pi)^{s_1+s_2}\Gamma(k)}\zeta(1+s_1+s_2)\frac{\sigma_{s_1-s_2}(n)n^{k-1}}{n^{s_1+k/2}} \\
    &\hspace{4mm}+2(2\pi)^{s_1+s_2}C_k^{-1}i^k\frac{n^{k-1}\sigma_{s_1-s_2}(n)}{n^{-s_2+k/2}}\frac{\Gamma(-s_1+k/2)\Gamma(-s_2+k/2)}{\Gamma(k)}\zeta(1-s_1-s_2)\\
    &\hspace{4mm}+2C_k^{-1}\frac{\Gamma(-s_1+k/2)\Gamma(s_2+k/2)}{(2\pi)^{-s_1+s_2}\Gamma(k)}\zeta(1-s_1+s_2)\frac{\sigma_{s_1+s_2}(n)n^{k-1}}{n^{s_2+k/2}}\\
    &\hspace{4mm}+2(2\pi)^{-s_1+s_2}C_k^{-1}\frac{n^{k-1}\sigma_{s_1+s_2}(n)}{n^{s_1+k/2}}\frac{\Gamma(s_1+k/2)\Gamma(-s_2+k/2)}{\Gamma(k)}\zeta(1+s_1-s_2).
\end{align*}
\end{prop}
\begin{proof}
    Notice that the region described in Proposition \ref{prop. singular} is contained in the absolutely convergent regions in Proposition \ref{prop. small cell} and \ref{prop. dual cell}. This yields the absolutely convergent region for $J_{\si}(\textbf{s},n).$ 
    
    Furthermore, each orbital integral has a meromorphic continuation to $\bC^2.$ So will $J_{\si}(\textbf{s},n).$ The explicit form is obtained by adding \eqref{main term 1}, \eqref{main term 3}, \eqref{main term 2} and \eqref{main term 4} together.
\end{proof}

\section{The regular orbital integral}

We introduce the following notations:
\[\bigsqcup_{m\neq0,-n}\bigsqcup_{\substack{ad=n+m\\ a,d\in\bZ}}\bigsqcup_{\substack{bc=m\\ b,c\in\bZ}}\begin{pmatrix}
        a&b\\c&d
    \end{pmatrix}=\Omega_1\bigsqcup \Omega_2\bigsqcup \Omega_3,\]
where
\begin{align*}
    \Omega_1&=\bigsqcup_{m\geq1}\bigsqcup_{\substack{ad=n+m\\ a,d\in\bZ}}\bigsqcup_{\substack{bc=m\\ b,c\in\bZ}}\begin{pmatrix}
        a&b\\c&d
    \end{pmatrix}\\
    \Omega_2&=\bigsqcup_{-n<m<0}\bigsqcup_{\substack{ad=n+m\\ a,d\in\bZ}}\bigsqcup_{\substack{bc=m\\ b,c\in\bZ}}\begin{pmatrix}
        a&b\\c&d
    \end{pmatrix}\\
    \Omega_3&=\bigsqcup_{m\leq-n-1}\bigsqcup_{\substack{ad=n+m\\ a,d\in\bZ}}\bigsqcup_{\substack{bc=m\\ b,c\in\bZ}}\begin{pmatrix}
        a&b\\c&d
    \end{pmatrix}
\end{align*}
Recall that $\omega=\begin{pmatrix}&-1\\1&\end{pmatrix}$, and we have the following relations:
\[\Omega_1\omega=\Omega_3,\hspace{10mm}\Omega_2\omega=\Omega_2,\hspace{10mm}\Omega_3\omega=\Omega_1\]
and hence
\[\bigsqcup_{m\neq0,-n}\bigsqcup_{\substack{ad=n+m\\ a,d\in\bZ}}\bigsqcup_{\substack{bc=m\\ b,c\in\bZ}}\begin{pmatrix}
        a&b\\c&d
    \end{pmatrix}=\Omega_1\omega\bigsqcup \Omega_2\omega\bigsqcup \Omega_3\omega.\]
Then set $\overline{\Omega}_j=\Omega_j/\sim$, where $g_1\sim g_2$ if $g_1=-g_2.$ Then we define:
\[J_{\Reg}^j(\textbf{s},n)=\frac{n^{k-1}}{C_k}\int_0^{\infty}\int_0^{\infty}\sum_{\gamma\in\Omega_j}\frac{y_1^{s_1+k/2}y_2^{s_2+k/2}}{R_{\gamma\omega}(iy_1,iy_2)^k}\frac{\,dy_1}{y_1}\frac{\,dy_2}{y_2}=\frac{2n^{k-1}}{C_k}\int_0^{\infty}\int_0^{\infty}\sum_{\gamma\in\overline{\Omega}_j}\frac{y_1^{s_1+k/2}y_2^{s_2+k/2}}{R_{\gamma\omega}(iy_1,iy_2)^k}\frac{\,dy_1}{y_1}\frac{\,dy_2}{y_2}.\]

The main result that we will prove in the following sections is:
\begin{prop}\label{prop. regular orbital integral}
    The regular orbital integral $J_{\Reg}(\textbf{s},n)$ is absolutely convergent in the region: 
    \begin{align*}
\begin{cases}
|\Re(s_1)|<\frac{k}{2}-1\\
|\Re(s_2)|<\frac{k}{2}-1. 
\end{cases}
\end{align*}
and
\[J_{\Reg}(\textbf{s},n)=J_{\Reg}^1(\textbf{s},n)+J_{\Reg}^2(\textbf{s},n)+J_{\Reg}^3(\textbf{s},n).\]
Furthermore, if we assume that $s_1-s_2$ is not an integer, then $J_{\Reg}^1(\textbf{s},n)$ $J_{\Reg}^2(\textbf{s},n)$ and $J_{\Reg}^3(\textbf{s},n)$ are explicitly given by \eqref{error term 1}, \eqref{error term 2} and \eqref{error term 3}.
\end{prop}
\begin{proof}
This is a direct corollary of Proposition \ref{prop. error 1}, Proposition \ref{prop. error 2} and Proposition \ref{prop. error 3}.
\end{proof}

\begin{remark}
    By the absolutely convergent region of the regular orbital integral, we can directly take $(s_1,s_2)=(0,0)$ and use Legendre functions and Legendre polynomials to analyze the regular orbital integrals. This approach is the main idea in \cite{WeiYangZhao2024} and will simply the calculations in the following several sections. However, our goal is to derive the result in Kuznetsov's preprint. Therefore, we will only assume that $(s_1,s_2)$ is in the absolutely convergent region, and we will assume $s_1-s_2$ is not an integer when calculating $J_{\Reg}^2(\textbf{s},n)$. 
\end{remark}

\subsection{The calculation of $J_{\Reg}^1(\textbf{s},n)$}

In this section, we prove the following result:
\begin{prop}\label{prop. error 1}
$J_{\Reg}^1(\textbf{s},n)$ is absolutely convergent in the region:
    \begin{align*}
\begin{cases}
|\Re(s_1)|<\frac{k}{2}-1\\
|\Re(s_2)|<\frac{k}{2}-1. 
\end{cases}
\end{align*}
Moreover, 
\begin{equation}\label{error term 1}
J_{\Reg}^1(\textbf{s},n)=\frac{4n^{k-1}}{C_kn^{k/2}}\sum_{m=1}^{\infty}\sigma_{s_1-s_2}(m)\sigma_{s_1+s_2}(n+m)\widetilde{\psi}_k\left({n},{m};s_1,s_2\right)
\end{equation}
where
\begin{align*}
\widetilde{\psi}_k\left({n},{m};s_1,s_2\right)&=\cos\left(\frac{\pi}{2}(s_1-s_2)\right)\frac{m^{s_2}}{(n+m)^{s_1+s_2}}\left(\frac{n}{m}\right)^{k/2}\\
&\hspace{10mm}\times  B(s_2+k/2,-s_2+k/2)B(-s_1+k/2,s_1+k/2)F(-s_2+k/2,-s_1+k/2;k;-n/m).
\end{align*}
\end{prop}
\begin{proof}
By \eqref{identiy for R} and change of variable $y_1\mapsto\frac{1}{y_1}$, we have:
\begin{align*}
J_{\Reg}^1(\textbf{s},n)&=2C_k^{-1}n^{k-1}i^k\int_0^{\infty}\int_0^{\infty}\sum_{\gamma\in\overline{\Omega}_1}\frac{y_1^{-s_1+k/2}y_2^{s_2+k/2}}{R_{\gamma}(iy_1,iy_2)^k}\frac{\,dy_1}{y_1}\frac{\,dy_2}{y_2}\\
&=2C_k^{-1}n^{k-1}i^k\sum_{\gamma\in\overline{\Omega}_1}\int_0^{\infty}\int_0^{\infty}\frac{y_1^{-s_1+k/2}}{j(\gamma,iy_1)^k}\frac{y_2^{s_2+k/2}}{(iy_2+\gamma(iy_1))^k}\frac{\,dy_2}{y_2}\frac{\,dy_1}{y_1}\\
&=2C_k^{-1}n^{k-1}i^k\sum_{m\geq1}\sum_{\substack{ad=n+m\\ a,d\geq1}}\sum_{\substack{bc=m\\b,c\geq1}}(W_{\gamma}(s_1,s_2)+\overline{W_{\gamma}(\overline{s_1},\overline{s_2})})
\end{align*}
where $\gamma=\begin{pmatrix}
    a&b\\c&d
\end{pmatrix}$ and
\[W(s_1,s_2)=\int_0^{\infty}\int_0^{\infty}\frac{y_1^{-s_1+k/2}}{j(\gamma,iy_1)^k}\frac{y_2^{s_2+k/2}}{(iy_2+\gamma(iy_1))^k}\frac{\,dy_2}{y_2}\frac{\,dy_1}{y_1}=\int_0^{\infty}\int_0^{\infty}\frac{y_1^{-s_1+k/2}}{j(\gamma,iy_1)^k\gamma(iy_1)^k}\frac{y_2^{s_2+k/2}}{(\beta y_2+1)^k}\frac{\,dy_2}{y_2}\frac{\,dy_1}{y_1},\]
with $\beta=\frac{i}{\gamma(iy_1)}.$ Notice that $\gamma\in \overline{\Omega}_j\subsetneq G_n$ and hence $\beta$ has nonzero imaginary part. 

Next, we study $W_{\gamma}(s_1,s_2)$ with $\gamma=\begin{pmatrix}
    a&b\\c&d
\end{pmatrix}$, $ad=n+m$ and $bc=m.$ We first consider the inner $y_2$-integral. By \cite[Equation 3.194(3)]{GradshteynRyzhik2007} (take $\mu=s_2+k/2$ and $\nu=k$), we obtain:
\[\int_0^{\infty}\frac{y_2^{s_2+k/2}}{(\beta y_2+1)^k}\frac{\,dy_2}{y_2}=\beta^{-s_2-k/2}B(s_2+k/2,-s_2+k/2).\]
Next we take $z_0=\frac{iay_1+b}{icy_1+d}$, $z_1=icy_1+d$ and $z_2=iay_1+d.$ We can show that $0<\arg(z_1)<\arg (z_2)<\frac{\pi}{2}$ and hence $0<\arg(z_0)<\frac{\pi}{2}.$ This yields:
\[\int_0^{\infty}\frac{y_2^{s_2+k/2}}{(\beta y_2+1)^k}\frac{\,dy_2}{y_2}=i^{-s_2-k/2}(iay_1+b)^{s_2+k/2}(icy_1+d)^{-(s_2+k/2)}B(s_2+k/2,-s_2+k/2).\]
Insert it into $W_{\gamma}(s_1,s_2)$ and we obtain:
\begin{align*}
W_{\gamma}(s_1,s_2)&=i^{-s_2-k/2}B(s_2+k/2,-s_2+k/2)\int_0^{\infty}y_1^{-s_1+k/2}(icy_1+d)^{-(s_2+k/2)}(iay_1+b)^{-(-s_2+k/2)}\frac{\,dy_1}{y_1}\\
&=\frac{i^{-s_2+k/2}}{c^{s_2+k/2}a^{-s_2+k/2}}B(s_2+k/2,-s_2+k/2)\int_0^{\infty}y_1^{-s_1+k/2}\left(y_1+\frac{d}{ic}\right)^{-(s_2+k/2)}\left(y_1+\frac{b}{ia}\right)^{-(-s_2+k/2)}\frac{\,dy_1}{y_1}
\end{align*}
By \eqref{integral repn for hypergeometric} (take $\nu=-s_1+k/2,$ $\mu=-s_2+k/2$, $\rho=s_2+k/2$, $\beta=\frac{b}{ia}$ and $\gamma=\frac{d}{ic}$ with $-\frac{k}{2}<\Re(s_1)<\frac{k}{2}$), $W(s_1,s_2)$ becomes:
\begin{align*}
    W_{\gamma}(s_1,s_2)=&e^{\frac{\pi i}{2}(s_1-s_2+k)}\frac{b^{s_2}c^{s_1}}{m^{k/2}d^{s_1+s_2}}B(s_2+k/2,-s_2+k/2)B(-s_1+k/2,s_1+k/2)\\
    &\hspace{10mm}\times F(-s_2+k/2,-s_1+k/2;k;-n/m).
\end{align*}
Insert $W_{\gamma}(s_1,s_2)$ into $J_{\Reg}^1(\textbf{s},n)$, and we obtain \eqref{error term 1}.

By \eqref{bound for tau function} and the fact $F(\alpha,\beta;\gamma;z)\ll 1$ as $z\to0,$ $J_{\Reg}^1(\textbf{s},n)$ is absolutely convergent in the region $|\Re(s_1)|<\frac{k}{2}-1$ and $|\Re(s_2)|<\frac{k}{2}-1.$
\end{proof}

\subsection{The calculation of $J_{\Reg}^3(\textbf{s},n)$}
The calculation of $J_{\Reg}^3(\textbf{s},n)$ is similar to that of $J_{\Reg}^1(\textbf{s},n).$ The main result is:
\begin{prop}\label{prop. error 3}
$J_{\Reg}^3(\textbf{s},n)$ is absolutely convergent in the region:
    \begin{align*}
\begin{cases}
|\Re(s_1)|<\frac{k}{2}-1\\
|\Re(s_2)|<\frac{k}{2}-1. 
\end{cases}
\end{align*}
Moreover, 
\begin{equation}\label{error term 3}
   J_{\Reg}^3(\textbf{s},n)=\frac{4n^{k-1}}{C_kn^{k/2}}\sum_{m\geq n+1}^{\infty}\sigma_{s_1-s_2}(m)\sigma_{s_1+s_2}(m-n)\widetilde{\Phi}_k\left({n},{m};s_1,s_2\right)
\end{equation}
where
\begin{align*}
\widetilde{\Phi}_k\left({n},{m};s_1,s_2\right)&=\cos\left(\frac{\pi}{2}(s_1+s_2)\right)\frac{m^{s_2}}{(m-n)^{s_1+s_2}}\left(\frac{n}{m}\right)^{k/2}\\
&\hspace{10mm}\times  B(s_2+k/2,-s_2+k/2)B(-s_1+k/2,s_1+k/2)F(-s_2+k/2,-s_1+k/2;k;n/m).
\end{align*}
\end{prop}
\begin{proof}
By \eqref{identiy for R} and change of variable $y_1\mapsto\frac{1}{y_1}$, $J_{\Reg}^2(\textbf{s},n)$ can be written as:
\[J_{\Reg}^3(\textbf{s},n)= 2C_k^{-1}n^{k-1}i^k\int_0^{\infty}\int_0^{\infty}\sum_{\gamma\in\overline{\Omega}_3}\frac{y_1^{-s_1+k/2}y_2^{s_2+k/2}}{R_{\gamma}(iy_1,iy_2)^k}\frac{\,dy_1}{y_1}\frac{\,dy_2}{y_2}.\]
By the definition of $\overline{\Omega}_3$, we obtain:
\begin{align*}
J_{\Reg}^2(\textbf{s},n)=2C_k^{-1}n^{k-1}i^k\sum_{m\geq n+1}\sum_{\substack{ad=m-n\\ a,d\geq1}}\sum_{\substack{bc=m\\b,c\geq1}}(V(s_1,s_2)+\overline{V(\overline{s_1},\overline{s_2})})
\end{align*}
where 
\[V(s_1,s_2)=V_{a,b,c,d}(s_1,s_2)=\int_0^{\infty}\int_0^{\infty}\frac{y_1^{-s_1+k/2}y_2^{s_2+k/2}}{(cy_1y_2-idy_2+iay_1+b)^k}\frac{\,dy_2}{y_2}\frac{\,dy_1}{y_1}.\]
We write
\[cy_1y_2-idy_2+iay_1+b=(iay_1+b)\left(\frac{cy_1-id}{iay_1+b}y_2+1\right)=(iay_1+b)(z_0y_2+1)\]
and
\[V(s_1,s_2)=\int_0^{\infty}\frac{y_1^{-s_1+k/2}}{(iay_1+b)^k}\int_0^{\infty}\frac{y_2^{s_2+k/2}}{(z_0y_2+1)^k}\frac{\,dy_2}{y_2}\frac{\,dy_1}{y_1}.\]
Notice that $\Im(z_0)<0.$ Notice that $\Re(z_0)>0.$ We first consider the inner $y_2$-integral. By \cite[Equation 3.194(3)]{GradshteynRyzhik2007} (take $\mu=s_2+k/2$ and $\nu=k$), we obtain:
\[\int_0^{\infty}\frac{y_2^{s_2+k/2}}{(z_0 y_2+1)^k}\frac{\,dy_2}{y_2}=z_0^{-s_2-k/2}B(s_2+k/2,-s_2+k/2).\]
Insert it into $W(s_1,s_2)$ and we obtain (notice that $\arg(iay_1+b),\arg(cy_1-id)\in(-\pi/2,\pi/2)$):
\begin{align*}
    V(s_1,s_2)&=B(-s_2+k/2,s_2+k/2)\int_0^{\infty}y_1^{-s_1+k/2}(cy_1-id)^{-(s_2+k/2)}(iay_1+b)^{-(-s_2+k/2)}\frac{\,dy_1}{y_1}\\
    &=B(-s_2+k/2,s_2+k/2)c^{-s_2-k/2}(ia)^{s_2-k/2}\\
    &  \hspace{10mm}\times\int_0^{\infty}y_1^{-s_1+k/2}\left(y_1-\frac{id}{c}\right)^{-(s_2+k/2)}\left(y_1+\frac{b}{ia}\right)^{-(-s_2+k/2)}\frac{\,dy_1}{y_1}
\end{align*}
By \eqref{integral repn for hypergeometric} (take $\nu=-s_1+\frac{k}{2}$ $\mu=-s_2+\frac{k}{2},$ $\rho=s_2+\frac{k}{2},$ $\beta=\frac{b}{ia}$ and $\gamma=-\frac{id}{c}$), we obtain:
\begin{align*}
V(s_1,s_2)&=c^{-s_2-k/2}(ia)^{s_2-k/2}\left(\frac{b}{ia}\right)^{s_2-k/2}\left(\frac{-id}{c}\right)^{-s_1-s_2}\\
&\hspace{10mm}\times B(-s_2+k/2,s_2+k/2)B(-s_1+k/2,s_1+k/2)F(k/2-s_2,k/2-s_1;k,n/m).
\end{align*}
Insert $V(s_1,s_2)$ into $J_{\Reg}^3(\textbf{s},n)$, and we obtain \eqref{error term 3}.

By \eqref{bound for tau function} and the fact $F(\alpha,\beta;\gamma;z)\ll 1$ as $z\to0,$ $J_{\Reg}^3(\textbf{s},n)$ is absolutely convergent in the region $|\Re(s_1)|<\frac{k}{2}-1$ and $|\Re(s_2)|<\frac{k}{2}-1.$
\end{proof}

\subsection{The calculation of $J_{\Reg}^2(\textbf{s},n)$} 
In this section, we prove:
\begin{prop}\label{prop. error 2}
   $J_{\Reg}^2(\textbf{s},n)$ is absolutely convergent in the region:
\begin{align*}
\begin{cases}
|\Re(s_1)|<\frac{k}{2}-1\\
|\Re(s_2)|<\frac{k}{2}. 
\end{cases}
\end{align*}
If we further assume that $s_1-s_2$ is not an integer, $J_{\Reg}^2(\textbf{s},n)$ can be expressed as:
\begin{equation}\label{error term 2}
J_{\Reg}^2(\textbf{s},n)=\frac{2n^{k-1}\pi}{C_kn^{k/2}}\sum_{1\leq m\leq n-1}\sigma_{s_1+s_2}(n-m)\sigma_{s_1-s_2}(m)\widetilde{\phi}_k(n,m;s_1,s_2),
\end{equation}
where
\begin{align*}
\widetilde{\phi}_k(n,m;s_1,s_2)&=\frac{n^{s_2}}{(n-m)^{s_1+s_2}}\frac{\Gamma(k/2-s_2)\Gamma(k/2+s_1)}{\sin\left(\frac{\pi}{2}(s_2-s_1)\right)\Gamma(1+s_1-s_2)\Gamma(k)}\times F(k/2-s_2,1-k/2-s_2;s_1-s_2+1;m/n) \\
&\hspace{2mm}+\frac{n^{s_1}m^{s_2-s_1}}{(n-m)^{s_1+s_2}}\frac{\Gamma(k/2+s_2)\Gamma(k/2-s_1)}{\sin\left(\frac{\pi}{2}(s_1-s_2)\right)\Gamma(1-s_1+s_2)\Gamma(k)}\times F(k/2-s_1,1-k/2-s_1;-s_1+s_2+1;m/n)\\
\end{align*}
\end{prop}

\begin{proof} 
By \eqref{identiy for R} and change of variable $y_1\mapsto\frac{1}{y_1}$, $J_{\Reg}^2(\textbf{s},n)$ can be written as:
\[J_{\Reg}^2(\textbf{s},n)= 2C_k^{-1}n^{k-1}i^k\int_0^{\infty}\int_0^{\infty}\sum_{\gamma\in\overline{\Omega}_2}\frac{y_1^{-s_1+k/2}y_2^{s_2+k/2}}{R_{\gamma}(iy_1,iy_2)^k}\frac{\,dy_1}{y_1}\frac{\,dy_2}{y_2}.\]
By the definition of $\overline{\Omega}_2$, we obtain:
\begin{equation}\label{regular term 2}
J_{\Reg}^2(\textbf{s},n)=2C_k^{-1}n^{k-1}i^k\sum_{1\leq m\leq n-1}\sum_{\substack{ad=n-m\\ a,d\geq1}}\sum_{\substack{bc=m\\b,c\geq1}}  U(s_1,s_2)+\overline{U(\overline{s_1},\overline{s_2})}
\end{equation}
where 
\[U(s_1,s_2)=U_{a,b,c,d}(s_1,s_2)=\int_0^{\infty}\int_0^{\infty}\frac{y_1^{-s_1+k/2}y_2^{s_2+k/2}}{(cy_1y_2+idy_2+iay_1+b)^k}\frac{\,dy_2}{y_2}\frac{\,dy_1}{y_1}.\]

Notice that this is a finite sum. Following the method in \cite[Proposition 2.2]{RR05}, $J_{\Reg}^2(\textbf{s},n)$ is absolutely convergent when $|\Re(s_1)|<\frac{k}{2}-1$ and $|\Re(s_2)|<\frac{k}{2}.$ In the following discussion, we assume that $|\Re(s_1)|<\frac{k}{2}-1$, $\Re(s_2)<-\frac{k}{2}+1$ and $s_1-s_2$ is not an integer.

We write
\[cy_1y_2+idy_2+iay_1+b=(iay_1+b)\left(\frac{cy_1+id}{iay_1+b}y_2+1\right)=(iay_1+b)(z_0y_2+1)\]
and
\[U(s_1,s_2)=\int_0^{\infty}\frac{y_1^{-s_1+k/2}}{(iay_1+b)^k}\int_0^{\infty}\frac{y_2^{s_2+k/2}}{(z_0y_2+1)^k}\frac{\,dy_2}{y_2}\frac{\,dy_1}{y_1}.\]
Notice that $\Re(z_0)>0.$ Now we fix $ad=n-m$ and $bc=m.$ We first consider the inner $y_2$-integral. By \cite[Equation 3.194(3)]{GradshteynRyzhik2007} (take $\mu=s_2+k/2$ and $\nu=k$), we obtain:
\[\int_0^{\infty}\frac{y_2^{s_2+k/2}}{(z_0 y_2+1)^k}\frac{\,dy_2}{y_2}=z_0^{-s_2-k/2}B(s_2+k/2,-s_2+k/2).\]

Insert it into $U(s_1,s_2)$ and we obtain (notice that $\arg(iay_1+b),\arg(cy_1+id)\in(0,\pi/2)$):
\begin{align*}
    U(s_1,s_2)&=B(-s_2+k/2,s_2+k/2)\int_0^{\infty}y_1^{-s_1+k/2}(cy_1+id)^{-(s_2+k/2)}(iay_1+b)^{-(-s_2+k/2)}\frac{\,dy_1}{y_1}\\
    &=\frac{B(-s_2+k/2,s_2+k/2)}{c^{s_2+k/2}(ia)^{-s_2+k/2}}\int_0^{\infty}y_1^{-s_1+k/2}\left(y_1+\frac{id}{c}\right)^{-(s_2+k/2)}\left(y_1+\frac{b}{ia}\right)^{-(-s_2+k/2)}\frac{\,dy_1}{y_1}
\end{align*}

\begin{remark}
    If we use the integral representation \eqref{integral repn for hypergeometric}, we will obtain $F(k/2-s_2,k/2-s_1;k,n/m)$. By our choice of $m,n$, we have $\frac{n}{m}>1$ and deriving \cite[Theorem 4.2]{BF21} is difficult.  Therefore, we will apply other integral representations for the hypergeometric functions.
\end{remark}

We set 
\begin{align*}
L(s_1,s_2)&=\int_0^{\infty}y_1^{-s_1+k/2}\left(y_1+\frac{id}{c}\right)^{-(s_2+k/2)}\left(y_1+\frac{b}{ia}\right)^{-(-s_2+k/2)}\frac{\,dy_1}{y_1}
\end{align*}

Let $R>0.$  We consider the sector contour constructed by the line segment connecting $0$ and $R,$ the arc connecting $R$ and $-iR$ and the line segment connecting $-iR$ and $0.$ Let $R\to\infty$ and apply Cauchy's theorem, which yields:
\begin{align*}
    L(s_1,s_2)
    &=\int_{0}^{\infty}(-iy_1)^{-s_1+k/2}\left(-iy_1+\frac{id}{c}\right)^{-(s_2+k/2)}\left(-iy_1-\frac{ib}{a}\right)^{-(-s_2+k/2)}\frac{\,dy_1}{y_1}\\
    &=i^{s_1+k/2}\int_{d/c}^{\infty}y_1^{-s_1+k/2}\left(y_1-\frac{d}{c}\right)^{-(s_2+k/2)}\left(y_1+\frac{b}{a}\right)^{-(-s_2+k/2)}\frac{\,dy_1}{y_1}\\
    &\hspace{10mm}+i^{s_1-2s_2-k/2}\int_{0}^{d/c}y_1^{-s_1+k/2}\left(\frac{d}{c}-y_1\right)^{-(s_2+k/2)}\left(y_1+\frac{b}{a}\right)^{-(-s_2+k/2)}\frac{\,dy_1}{y_1}\\
   &=i^{s_1+k/2}L_1(s_1,s_2)+i^{s_1-2s_2-k/2}L_2(s_1,s_2).
\end{align*}
For $L_1(s_1,s_2),$ take $\lambda=s_1-k/2+1,$ $\nu=s_2-k/2,$ $\mu=-s_2-k/2+1$, $\beta=\frac{b}{a}$ and $u=\frac{d}{c}$ and apply \eqref{integral repen for hyper, lower u}:
\[L_1(s_1,s_2)=\left(\frac{d}{c}\right)^{-s_1+k/2-1}\left(\frac{n}{ac}\right)^{-k+1}B(s_1+k/2,-s_2-k/2+1)F(s_1-k/2+1,-s_2-k/2+1;s_1-s_2+1;-m/(n-m))\]
The Pfaff transformation (\cite[Equation 9.131(1)]{GradshteynRyzhik2007}) implies:
\[F(\alpha,\beta;\gamma;z)=(1-z)^{-\beta}F(\beta,\gamma-\alpha;\gamma;z/(z-1))=(1-z)^{-\beta}F(\gamma-\alpha,\beta;\gamma;z/(z-1)).\]
The second equality is due to the definition of the hypergeometric functions (see \cite[Equation 9.100]{GradshteynRyzhik2007}).
This implies:
\begin{align*}
L_1(s_1,s_2)&=\left(\frac{d}{c}\right)^{-s_1+k/2-1}\left(\frac{n}{ac}\right)^{-k+1}\left(\frac{n}{n-m}\right)^{s_2+k/2-1}\\
&\hspace{10mm}\times B(s_1+k/2,-s_2-k/2+1)F(k/2-s_2,-s_2-k/2+1;s_1-s_2+1;m/n).
\end{align*}

For $L_2(s_1,s_2),$ take $\nu=-s_1+k/2,$ $\lambda=s_2-k/2,$ $\mu=-s_2-k/2+1$, $\alpha=\frac{b}{a}$ and $u=\frac{d}{c}$, and apply \eqref{integral repen for hyper, upper u}:
\[L_2(s_1,s_2)=\left(\frac{b}{a}\right)^{s_2-k/2}\left(\frac{d}{c}\right)^{-s_2-s_1}B(-s_2-k/2+1,-s_1+k/2)F\left(-s_2+k/2,-s_1+k/2;1-s_1-s_2;-\frac{n-m}{m}\right).\]
By \cite[Equation 9.132(1)]{GradshteynRyzhik2007}, we obtain:
\begin{align*}
   &F\left(-s_2+k/2,-s_1+k/2;1-s_1-s_2;-\frac{n-m}{m}\right)\\
   &\hspace{10mm}=\left(\frac{n}{m}\right)^{s_2-k/2}\frac{\Gamma(1-s_1-s_2)\Gamma(s_2-s_1)}{\Gamma(k/2-s_1)\Gamma(1-k/2-s_1)}F(k/2-s_2,1-k/2-s_2;s_1-s_2+1;m/n)\\
    &\hspace{20mm}+\left(\frac{n}{m}\right)^{s_1-k/2}\frac{\Gamma(1-s_1-s_2)\Gamma(s_1-s_2)}{\Gamma(k/2-s_2)\Gamma(1-k/2-s_2)}F(k/2-s_1,1-k/2-s_1;-s_1+s_2+1;m/n).
\end{align*}

Insert it into $L_2(s_1,s_2)$, which yields:
\begin{align*}
L_2(s_1,s_2)&=\left(\frac{b}{a}\right)^{s_2-k/2}\left(\frac{d}{c}\right)^{-s_2-s_1}\left(\frac{n}{m}\right)^{s_2-k/2}B(-s_2-k/2+1,s_2-s_1)\\
&\hspace{14mm}\times F(k/2-s_2,1-k/2-s_2;s_1-s_2+1;m/n)\\
&\hspace{4mm}+\left(\frac{b}{a}\right)^{s_2-k/2}\left(\frac{d}{c}\right)^{-s_2-s_1}\left(\frac{n}{m}\right)^{s_1-k/2}B(k/2-s_1,s_1-s_2)\\
&\hspace{14mm}\times F(k/2-s_1,1-k/2-s_1;-s_1+s_2+1;m/n)
\end{align*}
Notice that
\[\left(\frac{d}{c}\right)^{-s_1+k/2-1}\left(\frac{n}{ac}\right)^{-k+1}\left(\frac{n}{n-m}\right)^{s_2+k/2-1}=\left(\frac{b}{a}\right)^{s_2-k/2}\left(\frac{d}{c}\right)^{-s_2-s_1}\left(\frac{n}{m}\right)^{s_2-k/2}=\frac{(ac)^{s_1+k/2}}{(n-m)^{s_1+s_2}}n^{s_2-k/2},\]
and
\[\left(\frac{b}{a}\right)^{s_2-k/2}\left(\frac{d}{c}\right)^{-s_2-s_1}\left(\frac{n}{m}\right)^{s_1-k/2}=\frac{(ac)^{s_1+k/2}}{(n-m)^{s_1+s_2}}\frac{n^{s_1-k/2}}{m^{s_1-s_2}}.\]
By the definition of $L(s_1,s_2)$, $L_1(s_1,s_2)$ and $L_2(s_1,s_2),$ we obtain:
\begin{align*}
    L(s_1,s_2)&=\frac{(ac)^{s_1+k/2}}{(n-m)^{s_1+s_2}}n^{s_2-k/2}\widetilde{B}(s_1,s_2,k)F(k/2-s_2,1-k/2-s_2;s_1-s_2+1;m/n) \\
    &\hspace{2mm}+\frac{(ac)^{s_1+k/2}}{(n-m)^{s_1+s_2}}\frac{n^{s_1-k/2}}{m^{s_1-s_2}}i^{s_1-2s_2-k/2}B(k/2-s_1,s_1-s_2)F(k/2-s_1,1-k/2-s_1;-s_1+s_2+1;m/n)
\end{align*}
where
\[\widetilde{B}(s_1,s_2,k)=i^{s_1+k/2}B(s_1+k/2,-s_2-k/2+1)+i^{s_1-2s_2-k/2}B(-s_2-k/2+1,s_2-s_1).\]
Recall that
\[U(s_1,s_2)=\frac{B(-s_2+k/2,s_2+k/2)}{c^{s_2+k/2}(ia)^{-s_2+k/2}}L(s_1,s_2)\]
and this derives:
\begin{align*}
 U(s_1,s_2)&=\frac{a^{s_1+s_2}c^{s_1-s_2}}{(n-m)^{s_1+s_2}}n^{s_2-k/2}B(-s_2+k/2,s_2+k/2)\\
 &\hspace{16mm}\times i^{s_2-k/2}\widetilde{B}(s_1,s_2,k)F(k/2-s_2,1-k/2-s_2;s_1-s_2+1;m/n) \\
&\hspace{4mm}+\frac{a^{s_1+s_2}c^{s_1-s_2}}{(n-m)^{s_1+s_2}}\frac{n^{s_1-k/2}}{m^{s_1-s_2}}B(-s_2+k/2,s_2+k/2)\\
&\hspace{16mm}\times i^{s_1-s_2+k}B(k/2-s_1,s_1-s_2)F(k/2-s_1,1-k/2-s_1;-s_1+s_2+1;m/n)   
\end{align*}
It can be shown:
\begin{align*}
 i^{s_2-k/2}\widetilde{B}(s_1,s_2,k)&=i^{s_1+s_2+k}B(s_1+k/2,-s_2-k/2+1)+i^{s_1-s_2+k}B(-s_2-k/2+1,s_2-s_1)\\
 &=i^{s_1+s_2}\frac{\Gamma(k/2+s_1)\Gamma(1-k/2-s_2)}{\Gamma(1+s_1-s_2)}+i^{s_1-s_2+k}\frac{\Gamma(1-k/2-s_2)\Gamma(s_2-s_1)}{\Gamma(1-k/2-s_1)}\\
 &=i^{s_1+s_2+k}\frac{\Gamma(k/2+s_1)\Gamma(s_2-s_1)\sin(\pi(s_2-s_1))}{\Gamma(k/2+s_2)\sin(\pi s_2)}+i^{s_1-s_2+k}\frac{\Gamma(k/2+s_1)\Gamma(s_2-s_1)\sin(\pi s_1)}{\Gamma(k/2+s_2)\sin(\pi s_2)}
\end{align*}
The last equality is due to Euler's reflection formula. 

Then we obtain (apply Euler's reflection formula again):
\begin{align*}
 i^{s_2-k/2}\widetilde{B}(s_1,s_2,k)+\overline{i^{\overline{s_2}-k/2}}\overline{\widetilde{B}(\overline{s_1},\overline{s_2},k)}&=2i^k\cos\left(\frac{\pi}{2}(s_2-s_1)\right)\frac{\Gamma(k/2+s_1)\Gamma(s_2-s_1)}{\Gamma(k/2+s_2)}\\
 &=\pi i^k\frac{\Gamma(k/2+s_1)}{\sin\left(\frac{\pi}{2}(s_2-s_1)\right)\Gamma(1+s_1-s_2)\Gamma(k/2+s_2)}
\end{align*}
and hence (a third time Euler's reflection formula)
\begin{align*}
U(s_1,s_2)+\overline{U(\overline{s_1},\overline{s_2})}&=\pi i^k\frac{a^{s_1+s_2}c^{s_1-s_2}}{(n-m)^{s_1+s_2}}n^{s_2-k/2}\frac{B(-s_2+k/2,s_2+k/2)\Gamma(k/2+s_1)}{\sin\left(\frac{\pi}{2}(s_2-s_1)\right)\Gamma(1+s_1-s_2)\Gamma(k/2+s_2)}\\
&\hspace{16mm}\times F(k/2-s_2,1-k/2-s_2;s_1-s_2+1;m/n) \\
&\hspace{8mm}+\frac{a^{s_1+s_2}c^{s_1-s_2}}{(n-m)^{s_1+s_2}}\frac{n^{s_1-k/2}}{m^{s_1-s_2}}B(-s_2+k/2,s_2+k/2)B(k/2-s_1,s_1-s_2)\\
&\hspace{16mm}\times 2i^{k}\cos\left(\frac{\pi}{2}(s_1-s_2)\right)F(k/2-s_1,1-k/2-s_1;-s_1+s_2+1;m/n)\\
&=\pi i^k\frac{a^{s_1+s_2}c^{s_1-s_2}}{(n-m)^{s_1+s_2}}n^{s_2-k/2}\frac{\Gamma(k/2-s_2)\Gamma(k/2+s_1)}{\sin\left(\frac{\pi}{2}(s_2-s_1)\right)\Gamma(1+s_1-s_2)\Gamma(k)}\\
&\hspace{40mm}\times F(k/2-s_2,1-k/2-s_2;s_1-s_2+1;m/n) \\
&\hspace{8mm}+\pi i^k\frac{a^{s_1+s_2}c^{s_1-s_2}}{(n-m)^{s_1+s_2}}\frac{n^{s_1-k/2}}{m^{s_1-s_2}}\frac{\Gamma(k/2+s_2)\Gamma(k/2-s_1)}{\sin\left(\frac{\pi}{2}(s_1-s_2)\right)\Gamma(1-s_1+s_2)\Gamma(k)}\\
&\hspace{40mm}\times F(k/2-s_1,1-k/2-s_1;-s_1+s_2+1;m/n)\\
\end{align*}
Insert it into \eqref{regular term 2}, and we conclude
\[
J_{\Reg}^2(\textbf{s},n)=\frac{2n^{k-1}\pi}{C_kn^{k/2}}\sum_{1\leq m\leq n-1}\sigma_{s_1+s_2}(n-m)\sigma_{s_1-s_2}(m)\widetilde{\phi}_k(n,m;s_1,s_2),
\]
with $\widetilde{\phi}_k(n,m;s_1,s_2)$ defined in Proposition \ref{prop. error 2}.

Notice that this is a finite sum and it is absolutely convergent for all $(s_1,s_2)\in\bC^2$ provided that $s_1-s_2$ is not an integer. This completes the proof of Proposition \ref{prop. error 2}.
\end{proof}

\section{The weighted second moment}

In this section, we collect all the main terms (coming from $J_{\si}(\textbf{s},n)$) and the error terms (arising from $J_{\Reg}(\textbf{s},n)$). This results in Theorem \ref{main theorem, weighted moment}, which coincides with the result presented in Kuznetsov's preprint.

\begin{remark}
    It should be noted that we do not prove why $J_{\si}(\textbf{s},n)$ contributes the main term, while $J_{\Reg}(\textbf{s},n)$ contributes the error term. For the estimation part, readers can refer to \cite{BF21} or \cite{WeiYangZhao2024}.
\end{remark}

\subsection{The regularization}\label{sec. regularization}
In this section, we complete the ``regularization'' process. Without loss of generality, we assume that $k\geq12$, as we are only considering level $1$ holomorphic cusp forms. Then by Proposition \ref{prop. singular} and Proposition \ref{prop. regular orbital integral}, $J_{\Geo}(\textbf{s},n)$ is absolutely convergent in the region   
 \begin{align}\label{absolutely convergnet region}
\begin{cases}
\Re(s_1+s_2)>1, \hspace{2mm} \Re(-s_1+s_2)>1\\
|\Re(s_1)|<\frac{k}{2}-1, \hspace{2mm} |\Re(s_2)|<\frac{k}{2}-1,\\
\mbox{$s_1-s_2$ is not an integer}
\end{cases}
\end{align}
which nonempty since $k\geq12.$ Moreover, this region has a nonempty intersection with the absolutely convergent region of $J_{\Spec}(\textbf{s},n)$ ($\Re(s_1),\Re(s_2)>\frac{1}{2}.$)

On the other hand, by Proposition \ref{prop. singular} and Proposition \ref{prop. regular orbital integral}, $J_{\Geo}(\textbf{s},n)$ has a meromorphic continuation in the region \eqref{absolutely convergnet region}. Furthermore, all poles in this region are removable singularities. Therefore, by the standard complex analysis theorems, we can evaluate $J_{\Geo}(\textbf{s},n)$ in this region.
\subsection{Proof of main theorem and main corollary}\label{proof of main theorem and main corollary}
Recall that
\[J_{\Spec}(\textbf{s},n)=J_{\Geo}(\textbf{s},n)=J_{\si}(\textbf{s},n)+J_{\Reg}(\textbf{s},n).\]
By the discussion in \S \ref{sec. regularization}, the weighted moment \eqref{weighted second moment} is absolutely convergent in the region \eqref{absolutely convergnet region} with $(0,0)$ a removable singularity.

Then the proof of Theorem \ref{main theorem, weighted moment} follows immediately from \eqref{spectral side of RTF}, Proposition \ref{prop. singular} and Proposition \ref{prop. regular orbital integral}.

\begin{cor}\label{main corollary. reduce to 0}
Assume that $k\geq12.$ Let $\varepsilon>0$, and $C_{\varepsilon}:=\{z\in\bC:|z|=\varepsilon\}.$ Then
\[\frac{2\pi^2}{k-1}\sum_{f\in H_k}\frac{\lambda_f(n)L(1/2,f)^2}{L(1,\sym^2f)}=M_2(n;0,0)+E(n;0,0).\]
Here
\begin{align*}
 M_2(n;0,0)&=\frac{4}{n^{1/2}}\cdot\frac{1}{2\pi i}\oint_{C_{\varepsilon}}\frac{\Gamma((s+k)/2)^2}{(2\pi)^{s}\Gamma(k/2)^2}\frac{\sigma_0(n)}{n^{s/2}} \zeta(1+s)\frac{\,ds}{s},
\end{align*}
and
\begin{align*} 
E(n;0,0)&=2\frac{(-1)^{k/2}}{n^{1/2}}\frac{\Gamma(k/2)\Gamma(k/2)}{\Gamma(k)}\sum_{m=1}^{\infty}\sigma_{0}(m)\sigma_{0}(n+m)\left(\frac{n}{m}\right)^{k/2}F(-k/2,k/2;k;-n/m)\\
&\hspace{2mm}+2\frac{1}{n^{1/2}}\sum_{1\leq m\leq n-1}\sigma_{0}(n-m)\sigma_{0}(m)\phi_k\left(\frac{m}{n}\right)\\
&\hspace{2mm}+2\frac{(-1)^{k/2}}{n^{1/2}}\frac{\Gamma(k/2)\Gamma(k/2)}{\Gamma(k)}\sum_{m\geq n+1}^{\infty}\sigma_{0}(m)\sigma_{0}(m-n)\left(\frac{n}{m}\right)^{k/2}F(k/2,k/2;k;n/m)
\end{align*}
with
\[\phi_k(x)=\left(-\log x-2\frac{\Gamma'}{\Gamma}\left(\frac{k}{2}\right)+2\frac{\Gamma'}{\Gamma}(1)\right)F(k/2,1-k/2;1;x)-\left.\left(\frac{\partial}{\partial\alpha}+\frac{\partial}{\partial\beta}+2\frac{\partial}{\partial\gamma}\right)F(\alpha,\beta;\gamma;x)\right|_{\substack{\alpha=k/2\\ \beta=1-k/2\\ \gamma=1}}.\]
\end{cor}
\begin{proof}
By Theorem \ref{main theorem, weighted moment}, 
\[\frac{2\pi^2}{k-1}\sum_{f\in H_k}\frac{\lambda_f(n)L(1/2,f)^2}{L(1,\sym^2f)}=\lim_{s\to0}M_2(n;s,0)+\lim_{\substack{s_1\to0\\s_2\to0}}E(n;s_1,s_2):=M_2(n;0,0)+E(n;0,0).\]
The calculation of $M_2(n;0,0)$ is in \cite[Section 6]{WeiYangZhao2024}, and an explicit form can be found in \cite[Theorem B]{WeiYangZhao2024}.

For $E(n;0,0),$ the only difficult term is from $\phi_k(n,m;s_1,s_2),$ as the other two terms can be evaluated directly. The calculation of $\lim_{\substack{s_1\to0\\s_2\to0}}\phi_k(n,m;s_1,s_2)$ is in \cite[Section 5.1]{BF21}.
\end{proof}

\bibliographystyle{alpha}	
\bibliography{CRTF}

\end{document}